%% file: consist.tex
\documentclass[11pt]{article} 
\usepackage{enumitem}
\usepackage{amssymb, latexsym, amsmath}
\usepackage{mathtools}
\usepackage[dvipsnames]{xcolor}
\newtheorem{defin}{}
\newtheorem{saetze}[defin]{}
\newtheorem{conjec}[defin]{}
\newtheorem{lemmas}[defin]{}
\newtheorem{folger}[defin]{}
\newtheorem{bemerk}[defin]{}
\newtheorem{beispi}[defin]{}
\newtheorem{propos}[defin]{}

\newenvironment{definition}{\begin{defin}\it {\bf Definition:}}{\end{defin}}
\newenvironment{theorem}  {\begin{saetze}\it {\bf Theorem:}}{\end{saetze}}

\newenvironment{lemma}    {\begin{lemmas}\it {\bf Lemma:}}{\end{lemmas}}

\newenvironment{example}{\begin{beispi}\it {\bf Example:}}{\end{beispi}}
\newenvironment{proposition}{\begin{propos}\it {\bf Proposition:}}{\end{propos}}
\newenvironment{proof}    {{\it Proof}:}{{\hfill \fillbox \bigskip}}

\newcommand{\fillbox}{\mbox{$\bullet$}}
\newcommand{\ra}{\rightarrow}

\newcommand{\Lra}{\Leftrightarrow}
\newcommand{\ms}{\mapsto}

\newcommand{\N}{\mathbb N}

\newcommand{\Z}{\mathbb Z}

\renewcommand{\P}{\mathcal P}

\newenvironment{items}{\begin{list}{$\alph{item})$}
{\labelwidth18pt \leftmargin18pt \topsep3pt \itemsep1pt \parsep0pt}}
{\end{list}}

\begin{document}

\title{The word problem for 
       polycyclic groups and nilpotent associative algebras}
\author{Tobias Moede$^{1}$ and Matthias Neumann-Brosig$^{1,2}$}
\date{\today}
\maketitle

\begin{abstract}
The word problem is an old and central problem in (computational) group theory.
It is well-known that the word problem is undecidable in general, but
decidable for specific types of presentations. Consistent polycyclic 
presentations are an important class of group presentations with solvable
word problem. These presentations play a fundamental role in the algorithmic
theory of polycyclic groups. Problems analogous to the word 
problem arise when computing with other algebraic structures. Various aspects of 
this topic are considered in the literature. The aim of this paper is to provide 
a general approach to the topic including polycyclic groups and nilpotent associative 
algebras as examples.
\end{abstract}

\footnotetext[1]{Institut f\"ur Analysis und Algebra, Technische Universit\"at Braunschweig, Germany}
\footnotetext[2]{IAV GmbH, Gifhorn/Berlin, Germany}
\let\thefootnote\relax\footnotetext{E-mail addresses: t.moede@tu-braunschweig.de, m.neumann-brosig@tu-braunschweig.de}

\input{introduction}
\input{groups}
\input{algebras}
\input{summary}

\bibliographystyle{abbrv}
\bibliography{all}

\input{appendix}
\end{document}

%% file: introduction.tex
\section{Introduction}
\subsection{The word problem}
A \emph{presentation} of a group $G$ consists of a free group $F$, a set $X$ of free generators of $F$ together with a set of pairs
$R \subseteq F \times F$ such that $G \cong F/\langle \{ ab^{-1}\mid (a,b) \in R \} \rangle^F$ holds. This presentation is denoted by $\langle X \mid R \rangle$, 
and the elements of $R$ are called \emph{defining relations}. We often write a defining relation $(a,b)$ in the form $a=b$. A presentation $\langle X \mid R \rangle$ 
is called a \emph{finite presentation} if both $X$ and $R$ are finite sets; see \cite[Chapter 2]{Rob82} for more details.

\medskip
A fundamental problem in the algorithmic theory of finitely presented groups is the {\em word problem}, introduced by Dehn \cite{Deh11} in 1911 as the ``identity problem'': 
Given a finite presentation $\langle X \mid R \rangle$ of a group $G$ and
a word $w$ in $X$, is there an algorithm to decide
whether $w$ denotes the identity element in $G$, that is, $w \in \langle \{ ab^{-1}\mid(a,b) \in R \} \rangle^F$ holds? 
In the 1950s Novikov \cite{Nov55} and Boone \cite{MR101267} proved that the word problem is undecidable in general.
However, the word problem is decidable for specific types of presentations, for instance in the setting of polycyclic groups.  

\subsection{Polycyclic groups and polycyclic presentations}
A group $G$ is {\em polycyclic} if it has a {\em polycyclic series}; 
that is, a subnormal series with cyclic quotients. A \emph{polycyclic presentation} is a presentation on finitely many
generators 
$g_1, \ldots, g_n$, say, together with $r_1, \ldots, r_n \in \N \cup \{\infty\}$ and
integers $e_{i,k}, a_{i,j,k}, b_{i,j,k}$ for $1 \leq i < j \leq n$ and $1 \leq k \leq n$, with $0 \leq e_{i,k}, a_{i,j,k}, b_{i,j,k} < r_k$ 
for all $k$ such that $r_k < \infty$ and relations of the form
\begin{align*}
g_i^{r_i} &= g_{i+1}^{e_{i,i+1}} \cdots g_n^{e_{i,n}} 
              &&\mbox{ for } 1 \leq i \leq n \mbox{ with } r_i < \infty,  \tag{GR1} \label{GR1}\\
g_j g_i &= g_i g_{i+1}^{a_{i,j,i+1}} \cdots g_n^{a_{i,j,n}} 
              &&\mbox{ for } 1 \leq i < j \leq n, \mbox{ and } \tag{GR2} \\
g_j g_i^{-1} &= g_i^{-1} g_{i+1}^{b_{i,j,i+1}} \cdots g_n^{b_{i,j,n}} 
              &&\mbox{ for } 1 \leq i < j \leq n \mbox{ with } r_i = \infty. \tag{GR3}
\end{align*}
We denote the resulting presentation by $\P_r(e,a,b)$ and the group
defined by this presentation by $G_r(e,a,b)$.

From $\P_r(e,a,b)$, it is possible to calculate integers $c_{i,j,k},d_{i,j,k},f_{i,k}$ (with $0 \leq c_{i,j,k},d_{i,j,k},f_{i,k} < r_k$ for all $k$ such that $r_k < \infty$) such that the relations
\begin{align*}
 g_j^{-1} g_i &= g_i g_{i+1}^{c_{i,j,i+1}} \cdots g_n^{c_{i,j,n}} 
              &&\mbox{ for } 1 \leq i < j \leq n \mbox{ with } r_j = \infty \mbox{ and } \tag{GR4}\\
g_j^{-1} g_i^{-1} &= g_i^{-1} g_{i+1}^{d_{i,j,i+1}} \cdots g_n^{d_{i,j,n}} 
              &&\mbox{ for } 1 \leq i < j \leq n \mbox{ with } r_i = \infty, r_j = \infty \tag{GR5} \label{GR5}\\
g_i^{-1} &= g_i^{r_i-1} g_{i+1}^{f_{i,i+1}} \cdots g_n^{f_{i,n}} 
              &&\mbox{ for } 1 \leq i  \leq n \mbox{ with } r_i < \infty \tag{GR6} \label{GR6}
\end{align*}
hold in $G_r(e,a,b)$; see \cite[Chapter 9.4]{Sim94} for details.

Let $M$ be the set of words of finite length in the alphabet $\{g_1,...,g_n, g_1^{-1},...,g_n^{-1}\}$, with concatenation as multiplication, i.e. $M$ is the free monoid on
$\{g_1,...,g_n, g_1^{-1},...,g_n^{-1}\}$.
\emph{Collection} is a process that, given an input $w \in M$, substitutes subwords of $w$ equal to the left side of a relation \eqref{GR1} - \eqref{GR6} by the right side of the same relation. 
Also, any subword of the form $g_ig_i^{-1}$, $g_i^{-1}g_i$ can be replaced by the empty word.
If all of these replacements are impossible, the word is said to be in \emph{reduced form}. 

The presentation $\P_r(e,a,b)$ is called 
\emph{consistent} if and only if each element of $G_r(e,a,b)$ is represented by a unique reduced form. 
Collection for a polycyclic presentation terminates after finitely many steps in a reduced form, so if the presentation $\P_r(e,a,b)$ is consistent, the collection 
process is a solution of the word problem for $G_r(e,a,b)$. 

Consistency of polycyclic presentations has been studied by various authors.
Bayes, Kautsky \& Wamsley \cite{Wam74b}, Wamsley \cite{Wam74a}, Vaughan-Lee \cite{VL84},
Sims \cite[page 447]{Sim94} and Newman, Nickel \& Niemeyer \cite{NNN98} 
describe methods for nilpotent groups or, more specialised, for 
$p$-groups. Sims \cite[page 424]{Sim94} also exhibits a method for polycyclic
groups in general; this approach is based on monoid presentations and its
proof relies on the theory of confluent rewriting systems.

The \emph{collection to the left} algorithm always replaces 
the leftmost occurence of a non-reduced subword containing a generator with the lowest possible index (where the lower index has priority).
Thus, collection to the left induces an function $c\colon M \rightarrow M$ that maps $w$ 
to the reduced form of $w$ calculated by the collection to the left algorithm. We note that $c$ does not necessarily induce a function of the free group on $\{g_1,...,g_n\}$, as shown in Example \ref{exa::groups::eq_04_fails}.

For $x = (x_1, \ldots, x_n) \in \Z^n$ we write $g^x = g_1^{x_1} \cdots 
g_n^{x_n}$ to shorten notation. Similarly, we write $g^{a_{i,j}} = 
g_{i+1}^{a_{i,j,i+1}} \cdots g_n^{a_{i,j,n}}$ for $i < j$ as a short 
form for the right hand side of the first type of conjugate relation 
of $\P_r(e,a,b)$ and $g^{b_{i,j}}$ for the second type. Further, 
we define $g^{e_i} = g_{i+1}^{e_{i,i+1}} \cdots g_n^{e_{i,n}}$.

The first contribution of this paper is the following result, proven in Section \ref{subsec::groups}:

\begin{theorem}\label{thm::polpre}
The presentation $\P_r(e,a,b)$ is consistent if and only if the
following equations, called test equations, hold:
\begin{align*}
c(g_j g^{a_{j,k}} g_i) &= c(g_k g_j g_i) & (k > j > i)
    \label{eq::test1}\tag{G1}\\
c(g_j^{r_j} g_i)  &= c(g^{e_j} g_i) &(r_j < \infty, \, j > i) 
    \label{eq::test2}\tag{G2}\\
c(g_j g_i^{r_i}) &= c(g_j g^{e_i}) &(r_i < \infty, \, j > i)
    \label{eq::test3}\tag{G3}\\
c(g_j g_i^{-1} g_i) &= c(g_j) & (r_i = \infty, \, j > i)
    \label{eq::test4}\tag{G4}\\
c(g_i^{r_i+1}) &= c(g_i g^{e_i}) & (r_i < \infty)
    \label{eq::test5}\tag{G5}
\end{align*}
\end{theorem}
We note that the result is similar to that of Sims \cite[page 424]{Sim94}, which has a proof based on rewriting systems. 
In contrast, the proof Theorem \ref{thm::polpre} is purely group-theoretic, applying the theory of cyclic group extensions.

In case the polycyclic group is nilpotent, it is possible to check consistency on a subset of these equations if additional information in 
the form of a weight function is available and the relations have a certain form. The details of these nilpotent presentations are discussed 
in Section \ref{subsec::nilpotent}.

\subsection{Nilpotent associative algebras and nilpotent presentations}

Problems similar to the word problem play an important role for computing with other algebraic structures, e.g. nilpotent associative algebras 
over fields or the integers. As for groups, we are able to solve these problems for special types of presentations.

Let $R$ either be a field or the ring of integers. Then an associative $R$-algebra is an $R$-module
equipped with an associative, bilinear multiplication. For the remainder of this paper, we will 
assume all algebras to be associative.

A {\em presentation} for an algebra $A$ consists of a free algebra $F$, a set of free generators $X$ of $F$ together
with a set of pairs $R \subseteq F \times F$ such that $A \cong F/\langle\{ a-b \mid (a,b) \in R\} \rangle^F$ holds.
In analogy to the group case we denote such a presentation as $\langle X \mid R \rangle$ and call the elements of
$R$ {\em defining relations}, again often written as $a=b$. A presentation is called a {\em finite presentation}
if both $X$ and $R$ are finite sets.

For a finitely presented $R$-algebra $A$ the power ideal $A^i$ is the ideal generated by
all products of length $i$. Then $A$ is called nilpotent if there is $c\in \N$ with
$A = A^1 \geq A^2 \geq \ldots \geq A^c \geq A^{c+1} = \{0\}$.

A \emph{nilpotent presentation} for a finitely presented nilpotent algebra $A$ is a presentation
on finitely many generators $a_1, \ldots, a_n$, say, together with
$r_1, \ldots,r_n \in \N \cup \{ \infty \}$ and $e_{i,k}, b_{i,j,k} \in R$ for $1 \leq 
i < j \leq n$ and $1 \leq k \leq n$, with $0 \leq e_{i,k},  b_{i,j,k} < r_k$ for all $k$ such that $r_k < \infty$ and relations of the form
\begin{align*}
r_i a_i &= e_{i,i+1} a_{i+1} + \ldots + e_{i,n} a_n 
            &&\mbox{ for } 1 \leq i \leq n \mbox{ with } r_i < \infty, \tag{AR1} \label{AR1}\\
a_j a_i &= b_{i,j,\ell+1} a_{\ell+1} + \ldots + b_{i,j,n} a_n 
            &&\mbox{ for } 1 \leq i, j \leq n \mbox{ and } \ell = \max\{i,j\}. \tag{AR2} \label{AR2}
\end{align*}
We denote the resulting presentation by $\P_r(e,b)$ and the algebra
defined by this presentation by $A_r(e,b)$. If $R$ is a field, then the relations 
with left hand side $a_j a_i$ are sufficient to define $A_r(e,b)$. In this case we 
can choose $r_i = \infty$ for $1 \leq i \leq n$.

Collection for nilpotent algebras is a process that, given a linear combination of words
in the generators $\{a_1,\ldots,a_n\}$, produces a reduced form, i.e. a word in the generators
of the form $x_1 a_1 + \ldots + x_n a_n$ with $x_i \in R$ and $0 \leq x_i < r_i$ if
$r_i < \infty$. The presentation $\P_r(e,b)$ is called {\em consistent} if and only if each element of $A_r(e,b)$ 
is represented by a unique reduced form. 

Consistency of nilpotent presentations for nilpotent algebras has been studied in 
\cite{Eic07} for fields, and more explicitly for the integers in \cite{EiMo20}. In particular 
the implementation \cite{PKGzalgs} of the nilpotent quotient algorithm described in \cite{EiMo20}
uses consistency checks.

We describe a specific collection algorithm for nilpotent algebras in Section \ref{sec:naalgs}. This algorithm 
induces a function $c$ of the free algebra on $\{a_1,\ldots,a_n\}$ that maps an element to one of its reduced 
forms. Our second main result is the following theorem, which constitutes a consistency check for nilpotent algebras.

\begin{theorem} \label{algconst}
Suppose that $\P_r(e,b)$ is a nilpotent presentation and let $w$ be a
weight function for it. Then $\P_r(e,b)$ is consistent if and only if
the following equations, called test equations, hold: 
\begin{align*}
c(a_k c(a_j a_i)) &= c(c(a_k a_j) a_i) 
    & (w(a_k)+w(a_j)+w(a_i) \leq d)
    \label{nileqalg::test1}\tag{A1}\\    
c(r_j c(a_j a_i))  &= c( c(r_j a_j) a_i)
    & (r_j < \infty, \, w(a_j)+w(a_i) \leq d) 
    \label{nileqalg::test2}\tag{A2} \\
c(r_i c(a_j a_i )) &= c( a_j c( r_i a_i))
    & (r_i < \infty, \, w(a_j)+w(a_i) \leq d)
    \label{nileqalg::test3}\tag{A3}\\
\end{align*}
\end{theorem}

In the above theorem $w$ is a \emph{weight function} $w \colon \{a_1, \ldots, a_n\} \ra \N,\; a_i \ms w(a_i)$,
so that each image $w(a_i)$ is minimal subject to the conditions $w(a_k) \geq w(a_i)$ for each $i,k$ with 
$e_{i,k} \neq 0$, and $w(a_k) \geq w(a_i)+w(a_j)$ for each $i,j,k$ with $b_{i,j,k} \neq 0$.
Additionally, we define the weight of a reduced word in the generators as the minimum of the weights of the 
occuring generators. The empty word is assigned weight $\infty$, and we set $d = \max \{w(a_1),...,w(a_n)\}$.

%% file: groups.tex
\section{Consistency for polycyclic group presentations}

Given $G_r(e,a,b)$, we
define subgroups $G_i = \langle g_i, \ldots, g_n \rangle$ and 
indices $s_i = [G_i : G_{i+1}]$. Then $s_i \in \N \cup \{ \infty \}$
for $1 \leq i \leq n$. The sequence $s = (s_1, \ldots, s_n)$ are the
{\em indices} of $\P_r(e,a,b)$ and the sequence $r = (r_1, \ldots, r_n)$
are the {\em relative orders} of this presentation. The next lemma
follows readily from the relations of the presentation $\P_r(e,a,b)$.

\begin{lemma} \label{lem::polycyclic_series_group}
The series $G_r(e,a,b) = G_1 \geq G_2 \geq \ldots \geq G_n \geq 
G_{n+1} = \{1\}$ is a polycyclic series through $G$ satisfying $s_i =
[G_i:G_{i+1}] \leq r_i$ for $1 \leq i \leq n$. In particular, $G_r(e,a,b)$ 
is polycyclic.
\end{lemma}

\begin{definition}
The presentation $\P_r(e,a,b)$ is called {\em consistent} if $s_i = r_i$
holds for $1 \leq i \leq n$. Equivalently, $\P_r(e,a,b)$ is consistent if
and only if each element of $G_r(e,a,b)$ is represented by a unique reduced 
form. 
\end{definition}

\subsection{Nilpotent groups}
\label{subsec::nilpotent}
A finitely generated nilpotent group has a central series with cyclic 
quotients. Hence such a group is polcyclic and
it has a polycyclic series that is a central series. A nilpotent 
presentation reflects this: it is a polycyclic presentation such that 
the conjugate
relations have a special form: $g^{a_{i,j}} = g_j g_{j+1}^{a_{i,j,j+1}} \cdots
g_n^{a_{i,j,n}}$ and $g^{b_{i,j}} = g_j g_{j+1}^{b_{i,j,j+1}} \cdots 
g_n^{b_{i,j,n}}$. 

\begin{lemma}
If $\P_r(e,a,b)$ is a nilpotent presentation, then the polycyclic series
$G_r(e,a,b) = G_1 \geq G_2 \geq \ldots \geq G_n \geq G_{n+1} = \{1\}$ is
a central series through $G_r(e,a,b)$. In particular, $G_r(e,a,b)$ is
nilpotent.
\end{lemma}

We define a 
{\em weight function} $w \colon \{g_1, \ldots, g_n\} \ra \N,\, g_i \ms w(g_i)$,
so that each image $w(g_i)$ is minimal subject to the conditions
$w(g_k) \geq w(g_i)$ for each $i,k$ with $e_{i,k} \neq 0$, and $w(g_k) 
\geq w(g_i)+w(g_j)$ for each $i,j,k$ with $a_{i,j,k} \neq 0$ 
or $b_{i,j,k} \neq 0$.
Additionally, we define the weight of a reduced word in the generators 
as the minimum of the weights of the occuring generators. The empty
word gets the weight $\infty$. Let $d$ denote the maximum finite weight that
occurs; that is, $d = \max\{ w(g_i) \mid 1 \leq i \leq n\}$. Clearly,
two words commute if the sum of their weights is larger than $d$.
Hence Theorem \ref{thm::polpre} reduces to the following result.

\begin{theorem} \label{thm::nilpre}
Suppose that $\P_r(e,a,b)$ is a nilpotent presentation and let $w$ be a
weight function for it. Then $\P_r(e,a,b)$ is consistent if and only if
the following equations hold:
\begin{align*}
c(g_j g^{a_{j,k}} g_i) &= c(g_k g_j g_i) 
    & (k > j > i, w(g_k)+w(g_j)+w(g_i) \leq d)
    \\
c(g_j^{r_j} g_i)  &= c(g^{e_j} g_i) 
    & (r_j < \infty, \, j > i, \, w(g_j)+w(g_i) \leq d) 
    \\
c(g_j g_i^{r_i}) &= c(g_j g^{e_i}) 
    & (r_i < \infty, \, j > i, \, w(g_j)+w(g_i) \leq d)
    \\
c(g_j g_i^{-1} g_i) &= c(g_j) 
    & (r_i = \infty, \, j > i, \, w(g_j)+w(g_i) \leq d)
    \\
c(g_i^{r_i+1}) &= c(g_i g^{e_i}) 
    & (r_i < \infty, \, 2w(g_i) \leq d)
\end{align*}
\end{theorem}

In the case of finite $p$-groups, even more reductions are possible using so-called definitions of generators, which are a certain type of relation \cite{MR760652}. 
Since these definitions do not always exist for general nilpotent groups, we will not include them in our considerations.

\subsection{Cyclic extensions for groups}

The theory of cyclic extensions describes how a group can be constructed from
a normal subgroup and a cyclic factor group. Let $N$ be an arbitrary group, 
$C = \langle x \rangle$ be a cyclic group and $\alpha \colon N \ra N$ 
an automorphism. The following lemma determines group multiplications on the set $G = C \times N$. The
proof consists of direct calculations and is omitted.

\begin{lemma}\label{lem::cycextgrp}
$\mbox{}$
\begin{items}
\item[\rm (a)]
If $C$ is infinite, then $G = \{ (x^z, n) \mid z \in \Z, n \in N\}$ and
$(x^{z_1}, n_1) (x^{z_2}, n_2) = (x^{z_1+z_2}, \alpha^{z_2}(n_1) n_2)$ 
defines a group multiplication on $G$.
\item[\rm (b)]
If $C$ is finite of order $m$, then $G = \{ (x^z, n) \mid 0 \leq z < m,
n \in N\}$. Choose $y \in N$ and define
\[ (x^{z_1}, n_1) (x^{z_2}, n_2) = 
    \left \{ \begin{array}{ll}
     (x^{z_1+z_2}, \alpha^{z_2}(n_1) n_2) & \mbox{ if } z_1+z_2 < m \\
     (x^{z_1+z_2-m}, y \alpha^{z_2}(n_1) n_2) & \mbox{ if } z_1+z_2 \geq m 
    \end{array} \right.\]
Then this defines a group multiplication on $G$ if and only if 
$\alpha(y) = y$ and $\alpha^m$ is the inner automorphism of $N$ 
induced by conjugation with $y$.
\end{items}
\end{lemma}

\subsection{The proof of Theorem \ref{thm::polpre}}
\label{subsec::groups}
Recall that $\P_r(e,a,b)$ is a presentation on the generators $g_1, \ldots, 
g_n$ whose relations are defined by $e,a,b$. For $1 \leq k \leq n$ we define 
$\P^{[k]}_r(e,a,b)$ as the presentation on the generators $g_k, \ldots, g_n$
whose relations are those relations of $\P_r(e,a,b)$ which are words in these
generators only; that is, the defining relations of $\P^{[k]}_r(e,a,b)$ are 
\begin{eqnarray*}
g_i^{r_i} &=& g_{i+1}^{e_{i,i+1}} \cdots g_n^{e_{i,n}} 
              \mbox{ for } k \leq i \leq n \mbox{ with } r_i < \infty, \\
g_j g_i &=& g_i g_{i+1}^{a_{i,j,i+1}} \cdots g_n^{a_{i,j,n}} 
              \mbox{ for } k \leq i < j \leq n, \mbox{ and } \\
g_j g_i^{-1} &=& g_i^{-1} g_{i+1}^{b_{i,j,i+1}} \cdots g_n^{b_{i,j,n}} 
              \mbox{ for } k \leq i < j \leq n \mbox{ with } r_i = \infty.
\end{eqnarray*}
Let $F_k$ denote the free monoid on $\{g_k, \ldots, g_n, g_k^{-1}, \ldots, g_n^{-1}\}$. We interpret the collection of the left algorithm as a function $c: F_1 \rightarrow F_1$. 
It is necessary for the $F_k$ to be free monoids instead of free groups - if $F_1$ were a free group, the element represented by $g_i^{-1}g_i$ would be equal to the neutral element of the group, 
while replacing $g_i^{-1}g_i$ by the empty word is a non-trivial collection step.

Let $H^{[k]}$ 
be the group defined by $\P_r^{[k]}(e,a,b)$ and recall that $G_k$ is
the subgroup of $G$ generated by $g_k, \ldots, g_n$. Define the 
epimorphisms $\theta_k \colon F_k \ra H^{[k]},\, g_j \ms g_j (k \leq j \leq n)$ 
and $\rho_k \colon H^{[k]} \ra G_k,\, g_j \ms g_j (k \leq j \leq n)$ and
obtain

\begin{align*}
F_k \stackrel{\theta_k}{\longrightarrow}
H^{[k]}\stackrel{\rho_k}{\longrightarrow} G_k. 
\end{align*}

We note that neither of the epimorphisms $\theta_k$ or $\rho_k$ is an
isomorphism in general. For $\theta_k$ this is obvious, for $\rho_k$
we exhibit the following example.

\begin{example}
Let $\P_r(e,a,b) = \langle g_1,g_2 \mid g_2g_1=g_1g_2,\;g_2g_1^{-1}
=g_1^{-1}g_2^2 \rangle$. Then $H^{[2]} = \langle g_2 \mid \emptyset 
\rangle \cong \Z$ and $G_2$ is the trivial group, since $g_2 = 1$ in $G$.
\end{example}

The following lemma is a consequence of consistency. A proof can be found in \cite[p.~421]{Sim94}.

\begin{lemma} \label{lem::grpmono}
The presentation $\P_r(e,a,b)$ is consistent if and only if $\rho_i$ 
is an isomorphism for $1 \leq i \leq n$.
\end{lemma}

The presentation $\mathcal P^{[n]}_r(e,a,b)$ is always consistent, since 
it has only one generator $g_n$ and at most one relation. If the implication 
\[ \P^{[i+1]}_r(e,a,b) \mbox{ is consistent } 
   \Rightarrow  
   \P^{[i]}_r(e,a,b) \mbox{ is consistent } \]
follows from Equations \eqref{eq::test1}-\eqref{eq::test5} of Theorem 
\ref{thm::polpre}, then $\P_r(e,a,b)$ is consistent by induction. It is sufficient to show this implication in the 
case $i=1$ only.

\begin{lemma} 
Assume that $\P^{[2]}_r(e,a,b)$ is consistent. Then the map
\[\tau: F_2 \rightarrow F_2, g_j \mapsto g^{a_{1,j}}, g_j^{-1} \mapsto g^{c_{1,j}} (2 \leq j \leq n)\] induces an endomorphism of 
$H^{[2]}$ if and only if Equations \eqref{eq::test1} and \eqref{eq::test2} 
of Theorem \ref{thm::polpre} with $i=1$ are valid. 
\end{lemma}

\begin{proof}
``$\Rightarrow$'':
Assume that $\tau$ induces an endomorphism 
$\sigma$ of $H^{[2]}$. We show that all Equations 
\eqref{eq::test1} and \eqref{eq::test2} of Theorem \ref{thm::polpre} with 
$i=1$ are valid. 

In $H^{[2]}$ the relation $g_jg^{a_{j,k}} = g_kg_j$ holds for $2 \leq j < k
\leq n$. Since $\sigma$ is an endomorphism, it follows that 
$\sigma(g_jg^{a_{j,k}}) = \sigma(g_kg_j)$ holds in $H^{[2]}$. 
As $\P^{[2]}_r(e,a,b)$ is consistent, we obtain that the equations
\begin{align} \label{eq::proof_thm_grps_test_1}
c(\tau(g_jg^{a_{j,k}})) = c(\tau(g_kg_j)) \;\;\; (2 \leq j < k \leq n)
\end{align}
are valid in $F_2$. This yields in $F_1$:
\begin{eqnarray*}
&& c(g_j g^{a_{j,k}}g_1) \\	
&&=  c(g_1 \tau(g_jg^{a_{j,k}})) 
     \;\; \mbox{ by def of $\tau$ and collection}\\			
&&=  g_1 c(\tau(g_jg^{a_{j,k}}))
     \;\; \mbox{ via collection} \\
&&=  g_1c(\tau(g_kg_j))
     \;\; \mbox{ by Equation \eqref{eq::proof_thm_grps_test_1}} \\
&&= c(g_1\tau(g_kg_j))
     \;\; \mbox{ via collection} \\
&&= c(g_kg_jg_1)
      \;\; \mbox{ by def of $\tau$ and collection}.
\end{eqnarray*}

Thus Equations \eqref{eq::test1} of Theorem \ref{thm::polpre} are 
valid for $i = 1$. Similarly, since $\sigma(g_j^{r_j})=\sigma(g_j^{e_j})$ is a relation of $H^{[2]}$, it follows that the equations
\begin{align} 
c(\tau(g_j^{r_j})) = c(\tau(g_j^{e_j})) \;\;\; (2 \leq j <  n) \label{eq::proof_thm_grps_test_2}
\end{align}
are valid in $F_2$. This implies
\begin{align*}
 c(g_j^{r_j}g_1) 	&= c(g_1\tau(g_j^{r_j})) \;\; \mbox{ by def of $\tau$ and collection}\\
			&= g_1 c(\tau(g_j^{r_j}))\;\; \mbox{ via collection}  \\
			&= g_1 c(\tau(g_j^{e_j}))\;\; \mbox{ via Equation \eqref{eq::proof_thm_grps_test_2}} \\
			&= c(g_1 \tau(g_j^{e_j}))\;\; \mbox{ via collection} \\
			&= c(g_j^{e_j}g_1) \mbox{ by def of $\tau$ and collection.}
\end{align*}
Hence Equations \eqref{eq::test2} of Theorem \ref{thm::polpre} are 
valid as well for $i=1$. 
\medskip

``$\Leftarrow$'':
Assume that all Equations \eqref{eq::test1} and \eqref{eq::test2} of 
Theorem \ref{thm::polpre} are valid. We show that $\tau$ 
induces an endomorphism of $H^{[2]}$. 
$\tau$ induces an endomorphism of $H^{[2]}$ if all of the defining 
relations of $H^{[2]}$ are mapped to relations under $\tau$. There are three 
types of group relations: they have left hand sides $g_j^{r_j}$, $g_kg_j$ and 
$g_kg_j^{-1}$ for $2 \leq j < k \leq n$.

Additionally, since $\tau$ is defined to be the free monoid on $\{g_2,...,g_n,g_2^{-1},...,g_n^{-1}\}$, we need to check whether
$\tau(g_jg_j^{-1})$ and $ \tau(g_j^{-1}g_j)$ are mapped to relators of $H^{[2]}$. We have
\begin{align*}
c(\tau(g_jg_j^{-1})) &= c(g^{a_{1,j}}g^{c_{1,j}}) \text{ and}\\
c(\tau(g_j^{-1}g_j)) &= c(g^{c_{1,j}}g^{a_{1,j}}) 
\end{align*}
As described in the introduction, the exponents $c_{1,j,k}$ are calculated so that $g^{c_{1,j}}$ is inverse to $g^{a_{1,j}}$.
Since $\P^{[2]}_r(e,a,b)$ is consistent, there is only one reduced form of $(g^{a_{1,j}})^{-1}$, so the numbers $c_{1,j,k}$ are unique and the equations
\begin{align*}
c(g^{a_{1,j}}g^{c_{1,j}}) &= 1\\
c(g^{c_{1,j}}g^{a_{1,j}}) &= 1
\end{align*}
hold in $F_2$. This shows that $\tau(g_jg_j^{-1})$ and $\tau(g_j^{-1}g_j)$ are relators of $H^{[2]}$.

Consider the relation $g_kg_j = g_jg^{a_{j,k}}$ of $H^{[2]}$: 
\begin{eqnarray*}
&& g_1 c(\tau(g_jg^{a_{j,k}})) \\
 &&= c(g_1\tau(g_jg^{a_{j,k}}))
   \;\; \mbox{ by collection} \\
 &&= c(g_jg^{a_{j,k}}g_1)
   \;\; \mbox{ via def of $\tau$ and collection} \\
 &&= c(g_kg_jg_1)
   \;\; \mbox{ by Equation \eqref{eq::test1} with $i=1$} \\
 &&= c(g_1\tau(g_kg_j))
   \;\; \mbox{ via def of $\tau$ and collection} \\
 &&= g_1c(\tau(g_kg_j))
   \;\; \mbox{ by collection}
\end{eqnarray*}
Due to the fact that every collection step (on these words) corresponds 
to a relation of $H^{[2]}$, this shows that $\tau(g_jg^{a_{j,k}}) 
= \tau(g_kg_j)$ is a relation of $H^{[2]}$.

Next, consider the relation $g_j^{r_j}=g^{e_j}$ for $j \geq 2$: 
\begin{align*}
g_1c(\tau(g_j^{r_j})) &= c(g_1\tau(g_j^{r_j})) \;\; \mbox{ by collection} \\
		&= c(g_j^{r_j}g_1) \;\; \mbox{ via def of $\tau$ and collection}\\
		&= c(g^{e_j}g_1)\;\; \mbox{ by Equation \eqref{eq::test2} with $i=1$}\\
		&= c(g_1\tau(g^{e_j}))\;\; \mbox{ via def of $\tau$ and collection}\\
		&= g_1c(\tau(g^{e_j}))\;\; \mbox{ by collection.} 
\end{align*}
It follows that $\tau(g_j^{r_j}) = \tau(g^{e_j})$ is valid in $H^{[2]}$.

It remains to show that relations of the form $g_kg_j^{-1} 
= g_j^{-1}g^{b_{j,k}}$ with $ n\geq k > j > 1$ and $r_j = \infty$ are 
compatible with $\tau$. We prove that this relation is a consequence of 
the last two kinds of relations of $H^{[2]}$ considered previously. As 
$\P^{[2]}_r(e,a,b)$ is consistent, it follows that $\P^{[j]}_r(e,a,b)$
is consistent for $2 \leq j \leq n$. As $r_j=\infty$, the group 
$H^{[j]}$ is an infinite cyclic extension of $H^{[j+1]}$ via 
$\langle g_j \rangle$. Thus conjugation with $g_j$ induces an 
automorphism $\alpha \colon H^{[j+1]} \rightarrow H^{[j+1]}$ with $g_k=\alpha(x)$ for a unique $x \in H^{[j+1]}$. The relation 
$g_k^{g_j^{-1}} = g^{b_{j,k}}$ now follows from

\begin{align*}
g_k &= \alpha\left(\alpha^{-1}(g_k)\right) \\
 &= \alpha\left(g_k^{g_j^{-1}}\right)\\
 &= \alpha\left( x \right),\\
\end{align*}
Writing $x$ as a word in reduced form yields 
$g^{b_{j,k}}$, so we can calculate $g^{b_{j,k}}$ without using the 
relation $g_kg_j^{-1} = g_j^{-1}g^{b_{j,k}}$.

\end{proof}

\begin{lemma}
Assume that the presentation $\mathcal P^{[2]}_r(e,a,b)$ is consistent, $r_1 < \infty$ and the 
map $g_j \mapsto g^{a_{1,j}} (2 \leq j \leq n)$ induces an endomorphism $\sigma$ of $H^{[2]}$. 
\begin{items}
 \item[\rm (a)] If all Equations \eqref{eq::test3} of Theorem \ref{thm::polpre} are valid for $i=1$ and all $j > 1$, then $\sigma^{r_1}$ equals the inner 
 automorphism of $H^{[2]}$ induced by conjugation with $g^{e_1}$. In this case, $\sigma$ is an automorphism. 
 \item[\rm (b)] If, additionally, Equation \eqref{eq::test5} of Theorem \ref{thm::polpre} is valid for $i=1$, then $\sigma$ fixes $g^{e_1}$. 
\end{items}

\end{lemma}

\begin{proof}
Let $\tau: F_2 \rightarrow F_2, g_j \mapsto g^{a_{1,j}}$. $\tau$ induces the endomorphism $\sigma$ of $ H^{[2]}$.
\begin{items}
\item[\rm (a)] Applying Equation \eqref{eq::test3} of Theorem \ref{thm::polpre} with $i=1$, the definitions of $\tau$ and the collection algorithm yields in $F_1$:
\begin{eqnarray*}
&& c(g_jg^{e_1}) \\
 &&= c(g_jg_1^{r_1})
   \;\; \mbox{ by Equation \eqref{eq::test3} with $i=1$} \\
 &&= c(g_1^{r_1}\tau^{r_1}(g_j))
   \;\; \mbox{ via def of $\tau$ and collection} \\
 &&= c(g^{e_1}\tau^{r_1}(g_j))
   \;\; \mbox{ by collection}
\end{eqnarray*}
Since all collection steps on the words $g_jg^{e_1}$ and $g^{e_1}\tau^{r_1}(g_j)$ are derived from relations of $H^{[2]}$, 
it follows that $g_jg^{e_1} = g^{e_1}\sigma^{r_1}(g_j)$ is a valid equation in $H^{[2]}$. Thus, $\sigma^{r_1}$ is the inner automorphism induced by $g^{e_1}$. 
Since $\sigma^{r_1}$ is an epimorphism, it follows that $\sigma$ is an epimorphism. Since $H^{[2]}$ is polycyclic and therefore Hopfian, $\sigma$ is also an automorphism.

\item[\rm (b)] Using similar arguments as above, Equation \eqref{eq::test5} with $i=1$ implies via
\begin{eqnarray*}
&& g_1g^{e_1}\\
&&=c(g_1g^{e_1})  
\;\; \mbox{ by collection} \\
 &&=c(g_1^{r_1+1})
   \;\; \mbox{ by Equation \eqref{eq::test5} with $i=1$} \\
&&= c(g^{e_1}g_1)
   \;\; \mbox{ by collection} \\
&&= c(g_1 \tau(g^{e_1}))
   \;\; \mbox{ via def of $\tau$ and collection} \\
&&= g_1c( \tau(g^{e_1}))
   \;\; \mbox{ by collection}
\end{eqnarray*}
that $g^{e_1} = c(\tau(g^{e_1}))$ holds in $F_2$. Therefore, $\sigma$ fixes $g^{e_1}$ in $H^{[2]}$, since all collection steps on $\tau(g^{e_1})$ correspond to relations of $H^{[2]}$.
\end{items}
\end{proof}

\begin{lemma}
Assume that the presentation $\mathcal P^{[2]}_r(e,a,b)$ is consistent, $r_1 = \infty$ and the 
map $g_j \mapsto g^{a_{1,j}} (2 \leq j \leq n)$ induces an endomorphism $\sigma$ of $H^{[2]}$. If all Equations \eqref{eq::test4} of Theorem \ref{thm::polpre} with $i=1$ are valid, 
then the map $g_j \mapsto g^{b_{1,j}} (2 \leq j \leq n)$ induces the automorphism $\sigma^{-1}$ on $H^{[2]}$.
\end{lemma}

\begin{proof}
Let $\tau: F_2 \rightarrow F_2, g_j \mapsto g^{a_{1,j}}$ and $\xi: F_2 \rightarrow F_2, g_j \mapsto g^{b_{1,j}}$. 
Applying Equation \eqref{eq::test4} of Theorem \ref{thm::polpre} with $i=1$, the definitions of $\tau, \xi$ and the collection algorithm yields:
\begin{eqnarray*}
&& c(g_j)\\
&&=c(g_jg_1^{-1}g_1)  
\;\; \mbox{ by Equation \eqref{eq::test4} with $i=1$} \\
 &&=c(g_1^{-1}\xi(g_j)g_1)
   \;\; \mbox{ via def of $\xi$ and collection} \\
&&= c(g_1^{-1}g_1\tau(\xi(g_j)))
   \;\; \mbox{ via def of $\tau$ and collection} \\
&&= c(\tau(\xi(g_j)))
   \;\; \mbox{ by collection} 
\end{eqnarray*}
This shows that $\sigma$ (induced by $\tau$) is an epimorphism and thus an automorphism, since $H^{[2]}$ is Hopfian. 
The equation $c(g_j) = c(\tau(\xi(g_j)))$ then implies that $\xi$ induces the automorphism $\sigma^{-1}$ on $H^{[2]}$, 
because the collection algorithm only applies substitutions corresponding to relations of $H^{[2]}$ on the word $\tau(\xi(g_j))$.
\end{proof}

We are now in the position to complete the proof of Theorem \ref{thm::polpre}.\\

\begin{proof}
If the given presentation is consistent, all Equations \eqref{eq::test1} - \eqref{eq::test5} of Theorem \ref{thm::polpre} must hold, since
the equations contain words that represent the same elements in the group $G_1 =H^{[1]}$. 

Now assume that all Equations \eqref{eq::test1} - \eqref{eq::test5} of Theorem \ref{thm::polpre} are valid. The presentation $\mathcal P^{[n]}_r(e,a,b)$ is consistent. 
Assume by induction that the presentation $\mathcal P^{[i+1]}_r(e,a,b)$ is consistent. 
After renaming of indices, Equations \eqref{eq::test1} - \eqref{eq::test5} of Theorem \ref{thm::polpre} show that the group $H^{[i]}$ is an extension of $H^{{[i+1]}}$ via the preceding lemmas and Lemma \ref{lem::cycextgrp}. 
The presentation $\mathcal P^{[i]}_r(e,a,b)$ is then consistent by Lemma \ref{lem::grpmono}. It follows by induction on $i$ that the presentation $\mathcal P_r(e,a,b) = \mathcal P^{[1]}_r(e,a,b) $ is consistent.
\end{proof}

\subsection{Examples and irredundancy}
In this subsection we give, for each test equation, an example where that and only that sort of test equations fail. Details about the calculations can be found in Appendix \ref{appendixA}.

\begin{example} \label{ex::counter_group_1}
For the presentation 
\begin{align*}
\langle g_1,g_2,g_3 \mid &g_3g_1=g_1g_2,g_3g_2=g_2g_3^{-1},g_2g_1=g_1g_3 \\
&g_3{g_1^{-1}}={g_1^{-1}}g_2,g_3{g_2^{-1}}={g_2^{-1}}g_3^{-1},g_2{g_1^{-1}}={g_1^{-1}}g_3\rangle
\end{align*}
only Equation \eqref{eq::test1} fails.
\end{example}

\begin{example}
	For the presentation 
	\begin{align*}
	\langle g_1,g_2,g_3\mid &g_3^2=1, g_3g_2=g_2g_3, g_3{g_2^{-1}}={g_2^{-1}}g_3,\\
	&g_3{g_1}={g_1}g_2,g_3{g_1^{-1}}={g_1^{-1}}g_2,g_2{g_1}={g_1}g_3, g_2{g_1^{-1}}={g_1^{-1}}g_3 \rangle
	\end{align*} 
	only Equation \eqref{eq::test2} fails.
\end{example}

\begin{example}
	For the presentation
	\begin{align*}
	\langle g_1,g_2 \mid & g_2{g_1}={g_1}g_2^{-1}, g_1^3=1 \rangle
	\end{align*} 
	only Equation \eqref{eq::test3} fails. 
\end{example}

\begin{example}\label{exa::groups::eq_04_fails}
	For the presentation
	\begin{align*}
	\langle g_1,g_2 \mid & g_2{g_1}={g_1}g_2^2, g_2{g_1^{-1}}={g_1^{-1}}g_2 \rangle
	\end{align*} 
	only Equation \eqref{eq::test4} fails. 
\end{example}

\begin{example} \label{ex::counter_group_5}
	For the presentation
	\begin{align*}
	\langle g_1,g_2 \mid & g_2{g_1}={g_1}g_2^{-1},g_1^2=g_2 \rangle
	\end{align*} 
	only Equation \eqref{eq::test5} fails. 
\end{example}

Together, these examples prove the following proposition:

\begin{proposition}
The Equations \eqref{eq::test1} - \eqref{eq::test5} given in Theorem \ref{thm::polpre} are irredundant, that is, no proper subset of them is sufficient to detect inconsistency without additional assumptions on the presentation.
\end{proposition}

%% file: algebras.tex
\section{Consistency for nilpotent associative algebra presentations} \label{sec:naalgs}

\subsection{Nilpotent associative algebras}
Given $\P_r(e,b)$ and thus $A_r(e,b)$ we
define its ideals $A_i = \langle a_i, \ldots, a_n \rangle$ and 
the indices $s_i = [A_i : A_{i+1}]$. Then $s_i \in \N \cup \{ \infty \}$
for $1 \leq i \leq n$. The sequence $s = (s_1, \ldots, s_n)$ are the
{\em indices} of $\P_r(e,b)$ and the sequence $r = (r_1, \ldots, r_n)$
are the {\em relative orders} of this presentation. The next lemma
follows readily from the relations of the presentation $\P_r(e,b)$.

\begin{lemma}
If $\P_r(e,b)$ is a nilpotent presentation, then the series
$A_r(e,b) = A_1 \geq A_2 \geq \ldots \geq A_n \geq A_{n+1} = \{1\}$ is
a central series through $A_r(e,b)$. In particular, $A_r(e,b)$ is
a nilpotent algebra.
\end{lemma}

\subsection{Cyclic extensions for associative algebras}

We will briefly recall he theory of cyclic extensions for the case of associative algebras. Note that in the following lemma we omit the case of 
non-trivial multiplications on $C$, as we will focus on nilpotent associative algebras. The proof consists of direct calculations and will be omitted 
here.

\begin{lemma}\label{lem::cycextalg}
Let $R$ be a field or the rings of integers. Let $N$ be an associative $R$-algebra and let $C=\langle c \rangle$ be an associative
$R$-algebra with trivial multiplication. Let $s\in N$ and let $\varphi_L, \varphi_R\colon N \ra N$ be $R$-linear maps.

\begin{items}
 \item[\rm (a)] Let $C$ be infinite and for all $b,b' \in N$ let the following identities be satisfied:

\begin{items}
 \item[\rm (i)] $\varphi_L(s) = \varphi_R(s)$,
 \item[\rm (ii)] $\varphi_L^2(b) = sb$ and $\varphi_R^2(b) = bs$,
 \item[\rm (iii)] $\varphi_L(bb')=\varphi_L(b)b'$ and $\varphi_R(bb')=b\varphi_R(b')$,
 \item[\rm (iv)] $\varphi_R(b)b'=b \varphi_L(b')$ and $\varphi_L(\varphi_R(b))=\varphi_R(\varphi_L(b))$.
\end{items}
 
 Then the product \[ (\lambda c, b)(\lambda' c, b') = (0, \lambda \lambda' s + \lambda \varphi_L(b') + \lambda' \varphi_R(b) + bb'), \]
where $\lambda,\lambda' \in R$ and $b,b' \in N$, turns the direct sum of $R$-modules $C \oplus N$ into an associative $R$-algebra.

 \item[\rm (b)] Let $C$ be finite with $mc=0$ and $p\in N$. Furthermore, for all $b,b' \in N$ let {\rm (i)-(iv)} be satisfied and additionally:

\begin{items}
 \item[\rm (v)] $pb=m\varphi_L(b)$,
 \item[\rm (vi)] $bp=m\varphi_R(b)$.
\end{items}

 Then the addition \[ (\lambda c, b)+(\lambda' c, b') =     \left \{ \begin{array}{ll}
     ((\lambda+\lambda')c, b+b') & \mbox{ if } \lambda+\lambda' < m \\
     ((\lambda+\lambda'-m)c, p+b+b') & \mbox{ if } \lambda+\lambda' \geq m 
    \end{array} \right. \]
and the product \[ (\lambda c, b)(\lambda' c, b') = (0, \lambda \lambda' s + \lambda \varphi_L(b') + \lambda' \varphi_R(b) + b b'), \]
where $\lambda,\lambda' \in R$ and $b,b' \in N$, turn $C \times N$ into an associative $R$-algebra.
\end{items}
\end{lemma}

\subsection{The proof of Theorem \ref{algconst}}
Let $F_k$ denote the free associative algebra on $a_k, \ldots, a_n$, 
let $A^{[k]}$ be the algebra defined by $\P_r^{[k]}(e,b)$ and let $A_k$ be
the subalgebra of $A$ generated by $a_k, \ldots, a_n$. As in the group case
define epimorphisms $\theta_k \colon F_k \ra A^{[k]},\, a_j \ms a_j (k \leq j \leq n)$ 
and $\rho_k \colon A^{[k]} \ra A_k,\, a_j \ms a_j (k \leq j \leq n)$ to obtain

\begin{align*}
F_k \stackrel{\theta_k}{\longrightarrow}
A^{[k]}\stackrel{\rho_k}{\longrightarrow} A_k. 
\end{align*}

Again neither of the epimorphisms $\theta_k$ or $\rho_k$ are isomorphisms in 
general. We have the following lemma as a consequence of consistency.

\begin{lemma} \label{lem::algmono}
The presentation $\P_r(e,b)$ is consistent if and only if $\rho_i$ 
is an isomorphism for $1 \leq i \leq n$.
\end{lemma}

The presentation $\P^{[n]}_r(e,b)$ is always consistent, since 
it has only one generator $a_n$ and at most one relation. If the implication 
\[ \P^{[i+1]}_r(e,b) \mbox{ is consistent } 
   \Rightarrow  
   \P^{[i]}_r(e,b) \mbox{ is consistent } \]
follows from Equations \eqref{nileqalg::test1}-\eqref{nileqalg::test3} of Theorem 
\ref{algconst}, then $\P_r(e,b)$ is consistent by induction. As in the group case, 
it is sufficient to show this implication in the case $i=1$ only.

\medskip
Define the maps $\tau_L$ and $\tau_R$ by
\begin{align*}
 &\tau_L \colon F_2 \ra F_2,\;a_j \ms \sum \limits_{k=1}^n b_{j,1,k} a_k \mbox{\;\;\; and \;\;\;} \forall v,w \in F_2\colon \tau_L(vw)=\tau_L(v)w \\
 &\tau_R \colon F_2 \ra F_2,\;a_j \ms \sum \limits_{k=1}^n b_{1,j,k} a_k \mbox{\;\;\; and \;\;\;} \forall v,w \in F_2\colon \tau_R(vw)=v\tau_R(w) \\
\end{align*}

Furthermore, let
\[ s = \sum \limits_{k=1}^n b_{1,1,k} a_k, \]
and in the case $r_1 < \infty$, we also define
\[ p = \sum \limits_{k=1}^n e_{1,k} a_k. \]

\begin{lemma} \label{lem:nilalgindlinear}
Assume by induction that $\P^{[2]}_r(e,b)$ is consistent and Equations \eqref{nileqalg::test1} - \eqref{nileqalg::test3} hold, then $\tau_L$ and $\tau_R$ induce $R$-linear 
maps $\varphi_L, \varphi_R \colon A^{[2]} \ra A^{[2]}$.
\end{lemma}

\begin{proof}
We only give a proof for the case $\tau_L$ and $\varphi_L$ and note that an analogous proof works for the case $\tau_R$ and $\varphi_R$. Denote by $I$ the ideal such that the following sequence is exact:
\[ 0 \ra I \ra F_2 \ra A^{[2]} \ra 0. \]
Then $\tau_L$ induces an $R$-linear map $\varphi_L$ if and only if $\tau_L(I) \subseteq I$. As an ideal $I$ is generated by elements of the form
\[ a_j a_i - \sum \limits_{k=1}^n b_{i,j,k} a_k \mbox{\;\;\; and \;\;\;} r_i a_i - \sum \limits_{k=1}^n e_{i,k} a_k \]
for $i,j \in \{2,\ldots,n\}$. A crucial fact, which is used multiple times in the proof, is that the consistency of $\P^{[2]}_r(e,b)$ implies:
\[ \forall e,f \in F_2 : c(e) = c(f) \Lra e-f \in I. \;\;\;(\ast)\]

We then have
\begin{align*}
 \tau_L \left( a_j a_i - \sum \limits_{k=1}^n b_{i,j,k} a_k \right) &= \sum \limits_{\ell = 1}^n b_{j,1,\ell} a_\ell a_i - \sum \limits_{k=1}^n b_{i,j,k} \left( \sum \limits_{m=1}^n b_{k,1,m} a_m \right) \\
 &= \sum \limits_{\ell=1}^n b_{j,1,\ell} \left( \sum \limits_{t=1}^n b_{i,\ell,t} a_t \right) - \sum \limits_{k,m=1}^n b_{i,j,k} b_{k,1,m} a_m \\
 &= \sum \limits_{\ell,t=1}^n b_{j,1,\ell} b_{i,\ell,t} a_t - \sum \limits_{k,m=1}^n b_{i,j,k} b_{k,1,m} a_m \\
 &= \sum \limits_{\ell,t=1}^n \left( b_{j,1,\ell} b_{i,\ell,t} - b_{i,j,\ell} b_{\ell, 1, t} \right) a_t && (1)
\end{align*}

Next we use the fact that for all $1 \leq i,j,k \leq n$ the following holds:
\[ c(a_k c(a_j a_i)) = c(c(a_k a_j) a_i). \]
Evaluating both sides yields
\[ c(a_k c(a_j a_i)) = c \left( a_k \sum \limits_{\ell=1}^n b_{i,j,\ell} a_\ell \right) = c \left( \sum \limits_{\ell,t=1}^n b_{\ell,k,t} b_{i,j,\ell} a_t \right) \]
and
\[ c(c(a_k a_j) a_i) = c \left( \sum \limits_{\ell=1}^n b_{j,k,\ell} a_\ell a_i \right) = c \left( \sum \limits_{\ell,t=1}^n b_{j,k,\ell} b_{i,l,t} a_t \right). \]
Together with $(\ast)$ and with $k=1$ we deduce 
\[ \tau_L \left( a_j a_i - \sum \limits_{k=1}^n b_{i,j,k} a_k \right) = \sum \limits_{\ell,t=1}^n \left( b_{j,1,\ell} b_{i,\ell,t} - b_{i,j,\ell} b_{\ell, 1, t} \right) a_t \in I, \]

Next we consider the elements of the form, $r_ia_i - \sum_l e_{i,j} a_l$. We compute
\begin{align*}
\tau_L\left(r_ia_i - \sum_l e_{i,l} a_l\right) &= r_i\tau_L\left(a_i\right) - \sum_l e_{i,l} \tau_L\left(a_l\right)\\
&= \sum_k r_i b_{i,1,k}a_k - \sum_{l,m} e_{i,l} b_{l,1,m}a_m, && (2)
\end{align*}

Evaluating both sides of the consistency relations
\[ c(r_ic(a_1a_i)) = c(a_1c(r_ia_i)) \text{ for } 2 \leq i \leq n \]
we deduce
\[ c(r_ic(a_1a_i)) = c\left(r_i \sum_k b_{i,1,k}a_k\right) \mbox{\;\;\; and \;\;\;} c(a_1c(r_ia_i)) = c\left( \sum_{l,m} e_{i,l}b_{l,1,m}a_m \right). \]
Similar to the first case this allows us to deduce that
\[ \tau_L\left(r_ia_i - \sum_l e_{i,l} a_l\right) \in I. \]

Finally, note that a general element in $I$ is an $R$-linear combination of terms of the form $\sum_lx_lf_li_lg_l, \sum_lx_li_lg_l,\sum_lx_lf_li_l$ or $\sum_lx_li_l$,
where $f_l, g_l \in F_2, x_l \in R$ and $i_l \in I \cap \tau_L^{-1}(I)$. A straigtforward calculation shows that each of these linear combinations is mapped to
$I$ under $\tau_L$. Therefore $\varphi_L$ is induced by $\tau_L$.

\end{proof}

\begin{lemma} \label{lem::nilalgind}
Assume by induction that $\P^{[2]}_r(e,b)$ is consistent and Equations \eqref{nileqalg::test1} - \eqref{nileqalg::test3} hold, then the induced maps $\varphi_L, \varphi_R \colon A^{[2]} \ra A^{[2]}$ 
have the properties (i)-(vi) given in Lemma \ref{lem::cycextalg}.
\end{lemma}

\begin{proof}
It is sufficient to verify the properties given in Lemma \ref{lem::cycextalg} for the generators $a_1,\ldots,a_n$. We have
\begin{eqnarray*}
&& c(\tau_L(s)) \\
 &&= c(a_1 s)
   \;\; \mbox{ via def of $\tau_L$} \\
 &&= c(a_1 c(a_1 a_1))
   \;\; \mbox{ via def of $s$ and collection} \\
 &&= c(c(a_1 a_1) a_1)
   \;\; \mbox{ by Equation \eqref{nileqalg::test1}}  \\
 &&= c(s a_1)
   \;\; \mbox{ via def of $s$ and collection} \\
 &&= c(\tau_R(s))
   \;\; \mbox{ via def of $\tau_R$}
\end{eqnarray*}

As every collection step (on these words) corresponds to a relation of $A^{[2]}$, it follows that $\tau_L(s)=\tau_R(s)$ is a relation
of $A^{[2]}$. Hence $\varphi_L(s) = \varphi_R(s)$, i.e. (i) of Lemma \ref{lem::cycextalg}, holds. Note that by abuse of notation $s \in A^{[2]}$ is the image of $s \in F_2$
under the natural projection.

\medskip
To show (ii) of Lemma \ref{lem::cycextalg} we calculate
\begin{eqnarray*}
&& c(\tau_L^2(a_k)) \\
 &&= c(a_1 c(a_1 a_k))
   \;\; \mbox{ via def of $\tau_L$ and collection} \\
 &&= c(c(a_1 a_1) a_k)
   \;\; \mbox{ by Equation \eqref{nileqalg::test1}}  \\
 &&= c(s a_k)
   \;\; \mbox{ via def of $s$ and collection} 
\end{eqnarray*}

It follows that $\varphi_L^2(a_k)=sa_k$ holds for all $k \in \{2,\ldots,n\}$ and an analogous calculation shows that $\varphi_R^2(a_k)=a_k s$ holds for all $k\in \{2,\ldots,n\}$.

\medskip
To show (iii) of Lemma \ref{lem::cycextalg} we calculate
\begin{eqnarray*}
&& c(\tau_L(a_j a_k)) \\
 &&= c(a_1 c(a_j a_k))
   \;\; \mbox{ via def of $\tau_L$ and collection} \\
 &&= c(c(a_1 a_j) a_k)
   \;\; \mbox{ by Equation \eqref{nileqalg::test1}}  \\
 &&= c(\tau_L(a_j) a_k)
   \;\; \mbox{ via def of $\tau_L$ and collection} 
\end{eqnarray*}

It follows that $\varphi_L(a_j a_k)=\varphi_L(a_j) a_k$ holds for all $j,k \in \{2,\ldots,n\}$ and an analogous calculation shows that $\varphi_R(a_j a_k)=a_j \varphi_R(a_k)$ holds for all $j, k\in \{2,\ldots,n\}$.

\medskip
To show (iv) of Lemma \ref{lem::cycextalg} we calculate
\begin{eqnarray*}
&& c(\tau_R(a_j) a_k) \\
 &&= c(c(a_j a_1) a_k)
   \;\; \mbox{ via def of $\tau_R$ and collection} \\
 &&= c(a_j c(a_1 a_k))
   \;\; \mbox{ by Equation \eqref{nileqalg::test1}}  \\
 &&= c(a_j \tau_L(a_k))
   \;\; \mbox{ via def of $\tau_L$ and collection} 
\end{eqnarray*}

It follows that $\varphi_R(a_j) a_k= a_j\varphi_L(a_k)$ holds for all $j,k \in \{2,\ldots,n\}$. We also have
\begin{eqnarray*}
&& c(\tau_L(\tau_R(a_k)) \\
 &&= c(a_1 c(a_k a_1))
   \;\; \mbox{ via def of $\tau_L$,$\tau_R$ and collection} \\
 &&= c(c(a_1 a_k) a_1)
   \;\; \mbox{ by Equation \eqref{nileqalg::test1}}  \\
 &&= c(\tau_R(\tau_L(a_k))
   \;\; \mbox{ via def of $\tau_L$,$\tau_R$ and collection} 
\end{eqnarray*}
and $\varphi_L(\varphi_R(a_k)) = \varphi_R(\varphi_L(a_k))$ follows for all $k \in \{2,\ldots,n\}$.

\medskip
It remains to show $(v)$ and $(vi)$ in the case $r_1 < \infty$. To show $(v)$ we compute
\begin{eqnarray*}
&& c(r_1 \tau_L(a_k)) \\
 &&= c(r_1 c(a_1 a_k))
   \;\; \mbox{ via def of $\tau_L$ and collection} \\
 &&= c(c(r_1a_1)a_k)
   \;\; \mbox{ by Equation \eqref{nileqalg::test2}}  \\
 &&= c(pa_k)
   \;\; \mbox{ via def of $p$ and collection} 
\end{eqnarray*}
and $pa_k = r_1\varphi_L(a_k)$ follows for all $k \in \{2,\ldots,n\}$.

\medskip
To show $(vi)$ we compute
\begin{eqnarray*}
&& c(r_1 \tau_R(a_k)) \\
 &&= c(r_1 c(a_k a_1))
   \;\; \mbox{ via def of $\tau_R$ and collection} \\
 &&= c(a_k c(r_1a_1))
   \;\; \mbox{ by Equation \eqref{nileqalg::test3}}  \\
 &&= c(a_k p)
   \;\; \mbox{ via def of $p$ and collection} 
\end{eqnarray*}
and $a_k p = r_1\varphi_R(a_k)$ follows for all $k \in \{2,\ldots,n\}$.
\end{proof}

We now give the proof of Theorem \ref{algconst}.\\

\begin{proof}
If the given presentation is consistent, all Equations \eqref{nileqalg::test1} - \eqref{nileqalg::test3} 
of Theorem \ref{algconst} must hold, since the equations contain words that represent the same elements in the 
algebra $A_1 = A^{[1]}$. 

Now assume that all Equations \eqref{nileqalg::test1} - \eqref{nileqalg::test3} of Theorem \ref{algconst} are valid. 
The presentation $\P^{[n]}_r(e,b)$ is consistent. Assume by induction that the presentation $\P^{[i+1]}_r(e,b)$ is consistent. 
After renaming of indices, Lemmas \ref{lem:nilalgindlinear} and \ref{lem::nilalgind}, together with Lemma \ref{lem::cycextalg} shows that $A^{[i]}$ is an extension of $A^{{[i+1]}}$.
The presentation $\P^{[i]}_r(e,b)$ is then consistent by Lemma \ref{lem::algmono}. It follows by induction on $i$ that the 
presentation $\P_r(e,b) = \P^{[1]}_r(e,b) $ is consistent.

\subsection{Examples and irredundancy}
Similar to the group case we give, for each test equation, an example where precisely that sort of test equations fail.
Detailed calculations can be found in Appendix \ref{appendixB}.

\begin{example} \label{ex::counter_alg_1}
For the presentation
\[ \langle a_1, a_2, a_3 \mid a_1^2 = a_2, a_1a_2=a_3, \mbox{ all other products trivial } \rangle \]
only Equations of the form \eqref{nileqalg::test1} fail. 
\end{example}

\begin{example} \label{ex::counter_alg_2}
For the presentation
\[ \langle a_1, a_2, a_3 \mid 2a_1 = 0, a_1 a_2 = a_3, \mbox{ all other products trivial } \rangle \]
only Equations of the form \eqref{nileqalg::test2} fail.
\end{example}

\begin{example} \label{ex::counter_alg_3}
For the presentation
\[ \langle a_1, a_2, a_3 \mid 2a_1 = 0, a_2 a_1 = a_3, \mbox{ all other products trivial } \rangle \]
only Equations of the form \eqref{nileqalg::test3} fail. 
\end{example}

Together, these examples prove the following proposition:

\begin{proposition}
The Equations \eqref{nileqalg::test1} - \eqref{nileqalg::test3} given in Theorem \ref{algconst} are irredundant, that is, no proper subset of them is sufficient to detect 
inconsistency without additional assumptions on the presentation.
\end{proposition}

\end{proof}

%% file: summary.tex
\section{Summary and Outlook}
We have demonstrated the connection between the theories of cyclic extensions and consistent presentations 
for the cases of polycyclic groups and nilpotent associative algebras. Similar ideas can be applied to 
other categories, like Lie rings, Leibniz algebras, Malcev algebras or Jordan algebras.

The general construction of an algorithm that responds to an input (a suitable description of an algebraic structure) with a complete and irredundant 
set of test equations is still an open problem.

%% file: appendix.tex
\appendix
\section{Calculations for Examples \ref{ex::counter_group_1} - \ref{ex::counter_group_5}} \label{appendixA}
In this section we show the calculations behind Examples \ref{ex::counter_group_1} - \ref{ex::counter_group_5}. 
The examples are restated here for convenience.

\setcounter{defin}{12}

\begin{example} 
For the presentation 
\begin{align*}
\langle g_1,g_2,g_3 \mid &g_3g_1=g_1g_2,g_3g_2=g_2g_3^{-1},g_2g_1=g_1g_3 \\
&g_3{g_1^{-1}}={g_1^{-1}}g_2,g_3{g_2^{-1}}={g_2^{-1}}g_3^{-1},g_2{g_1^{-1}}={g_1^{-1}}g_3\rangle
\end{align*}
only Equation \eqref{eq::test1} fails, since the presentation $\mathcal P^{[2]}_r(e,a,b)$ is a consistent presentation for a semidirect product $\mathbb Z \rtimes \mathbb Z$. 
Any failing test equation must thus include $g_1$, which leaves Equations \eqref{eq::test1} and \eqref{eq::test4} (since all $r_i$ are infinite). Direct computation shows: 
\begin{align*}
c(g_3 g_2 g_1) &= g_1g_2g_3\\
c(g_2g_3^{-1}g_1) &= c(g_1g_3g_2^{-1})\\
&= g_1g_2^{-1}g_3^{-1},
\end{align*}
so Equation \eqref{eq::test1} fails, and

\begin{align*}
	c(g_3g_1^{-1}g_1) &= c(g_1^{-1}g_2g_1)\\
	&=c(g_1^{-1}g_1g_3)\\
	&=g_3
\end{align*}
and
\begin{align*}
	c(g_2g_1^{-1}g_1) &= c(g_1^{-1}g_3g_1)\\
	&=c(g_1^{-1}g_1g_2)\\
	&=g_2,
\end{align*} 
therefore both test Equations \eqref{eq::test4} are valid.
\end{example}

\begin{example}
	For the presentation 
	\begin{align*}
	\langle g_1,g_2,g_3\mid &g_3^2=1, g_3g_2=g_2g_3, g_3{g_2^{-1}}={g_2^{-1}}g_3,\\
	&g_3{g_1}={g_1}g_2,g_3{g_1^{-1}}={g_1^{-1}}g_2,g_2{g_1}={g_1}g_3, g_2{g_1^{-1}}={g_1^{-1}}g_3 \rangle
	\end{align*} 
	only Equation \eqref{eq::test2} fails: again, the corresponding group $H^{[2]}$ has a consistent presentation 
	(it is isomorphic to $\mathbb Z \times (\mathbb Z/(2 \mathbb Z)$), so if a test equation fails, 
	it must fail on a pair of words containing $g_1$. We check Equation \eqref{eq::test1}:
	\begin{align*}
	c(g_3g_2g_1) &= g_1g_2g_3\\
	c(g_2g_3g_1) &= c(g_1g_3g_2)\\
	&=g_1g_2g_3.
	\end{align*}
	Test Equation \eqref{eq::test2} fails for $i=1,j=3$:
	\begin{align*}
	c(g_3^2g_1) &= g_1g_2^2\\
	c(g_1) &= g_1
	\end{align*}
	The only other type of test Equation that can fail is Equation \eqref{eq::test4}:
	\begin{align*}
	c(g_3g_1^{-1}g_1) &= c(g_1^{-1}g_2g_1)\\
	&= c(g_1^{-1}g_1g_3)\\
	&=g_3
	\end{align*}
	\begin{align*}
	c(g_2g_1^{-1}g_1) &= c(g_1^{-1}g_3g_1)\\
	&= c(g_1^{-1}g_1g_2)\\
	&=g_2,
	\end{align*}
therefore Equation \eqref{eq::test4} is valid for all admissible pairs of indices.
\end{example}

\begin{example}
	For the presentation
	\begin{align*}
	\langle g_1,g_2 \mid & g_2{g_1}={g_1}g_2^{-1}, g_1^3=1 \rangle
	\end{align*} 
	only test Equation \eqref{eq::test3} fails. There is only one pair of indices for Equations \eqref{eq::test3} and \eqref{eq::test5} to check. Equation \eqref{eq::test3} fails:
	\begin{align*}
	c(g_2g_1^3) &= c(g_1^3g_2^{-1})\\
	&= g_2^{-1}\\
	&\neq c(g_2)
	\end{align*}
	Equation \eqref{eq::test5} does not fail:
	\begin{align*}
	c(g_1^4) &= g_1.
	\end{align*}
\end{example}

\begin{example}
	For the presentation
	\begin{align*}
	\langle g_1,g_2 \mid & g_2{g_1}={g_1}g_2^2, g_2{g_1^{-1}}={g_1^{-1}}g_2 \rangle
	\end{align*} 
	only test Equation \eqref{eq::test4} fails. Since there are only two generators and no finite $r_i$, Equation \eqref{eq::test4} is the only one to check:
	\begin{align*}
	c(g_2g_1^{-1}g_1) &= c(g_1^{-1}g_2g_1)\\
	&=c(g_1^{-1}g_1g_2^2)\\
	&=g_2^2\\
	&\neq g_2.
	\end{align*}
\end{example}

\begin{example}
	For the presentation
	\begin{align*}
	\langle g_1,g_2 \mid & g_2{g_1}={g_1}g_2^{-1},g_1^2=g_2 \rangle
	\end{align*} 
	only Equation \eqref{eq::test5} fails. The only test equations defined are Equation \eqref{eq::test3} and \eqref{eq::test5}. The unique Equation \eqref{eq::test3} does not fail:
	\begin{align*}
	c(g_2g_1^2) &=c(g_1^2g_2)\\
	&=g_2^2\\
	c(g_2^2) &= g_2^2.
	\end{align*}
	The unique Equation \eqref{eq::test5} fails:
	\begin{align*}
	c(g_1^3) &=c(g_2g_1)\\
	&=g_1g_2^{-1}\\
	c(g_1g_2) &= g_1g_2.
	\end{align*}
\end{example}

\section{Calculations for Examples \ref{ex::counter_alg_1} - \ref{ex::counter_alg_3}} \label{appendixB}
\setcounter{defin}{23}
\begin{example}
For the presentation
\[ \langle a_1, a_2, a_3 \mid a_1^2 = a_2, a_1a_2=a_3, \mbox{ all other products trivial } \rangle \]
only Equations of the form \eqref{nileqalg::test1} fail. We have $w(a_1)=1, w(a_2)=2, w(a_3)=3$ and $r_1=r_2=r_3=\infty$. We compute
\[ c(a_1 c(a_1 a_1)) = c(a_1 a_2) = a_3 \neq 0 = c(a_2 a_1) = c(c(a_1 a_1) a_1). \]
There are no other test equations in this case.
\end{example}

\begin{example}
For the presentation
\[ \langle a_1, a_2, a_3 \mid 2a_1 = 0, a_1 a_2 = a_3, \mbox{ all other products trivial } \rangle \]
only Equations of the form \eqref{nileqalg::test2} fail. We have $w(a_1)=w(a_2)=1, w(a_3)=2$ and $r_1=2,r_2=r_3=\infty$. We compute
\[ c(2 c(a_1 a_2)) = c(2 a_3) = 2a_3 \neq 0 = c(0 \cdot a_2) = c(c(2 a_1) a_2). \]
In this case there are no test equations of the form \eqref{nileqalg::test1} and there are two test equations of the form \eqref{nileqalg::test3}.
For these we compute
\[ c(2 c(a_1 a_1)) = 0 = c(a_1 c(2a_1)) \]
and 
\[ c(2 c(a_2 a_1)) = c(2 \cdot 0) = 0 = c(a_2 \cdot 0) = c(a_2 c(2 a_1)). \]
\end{example}

\begin{example}
For the presentation
\[ \langle a_1, a_2, a_3 \mid 2a_1 = 0, a_2 a_1 = a_3, \mbox{ all other products trivial } \rangle \]
only Equations of the form \eqref{nileqalg::test3} fail. We have $w(a_1)=w(a_2)=1, w(a_3)=2$ and $r_1=2,r_2=r_3=\infty$. We compute
\[ c(2 c(a_2 a_1)) = c(2 a_3) = 2a_3 \neq 0 = c(a_2 \cdot 0) = c(a_2 c(2 a_1)). \]
In this case there are no test equations of the form \eqref{nileqalg::test1} and there are two test equations of the form \eqref{nileqalg::test2}.
For these we compute
\[ c(2 c(a_1 a_1)) = 0 = c(c(2a_1)a_1) \]
and 
\[ c(2 c(a_1 a_2)) = c(2 \cdot 0) = 0 = c(0 \cdot a_2) = c(c(2 a_1) a_2). \]
\end{example}